%% file: acylindrically-hyperbolic-sr1.tex
\documentclass[a4paper,leqno,12pt]{article} 

\usepackage[utf8]{inputenc}
\usepackage[T1]{fontenc}
\usepackage[english]{babel}

\usepackage{crimson}

\usepackage{geometry}

\usepackage{fancyhdr}

\usepackage{hyperref}
\hypersetup{
  colorlinks   = true, 
  urlcolor     = blue, 
  linkcolor    = blue, 
  citecolor   = black 
}

\usepackage[inline]{enumitem}

\usepackage{graphicx}

\usepackage{color}

\usepackage[intlimits]{amsmath}
\usepackage{amssymb, amsfonts}

\usepackage[thmmarks, amsmath, amsthm]{ntheorem}

\usepackage{MnSymbol}
\usepackage{mathbbol}

\usepackage{bbm}

\usepackage{mathrsfs}

\usepackage[all]{xy}

\numberwithin{equation}{section}
\newtheorem{theoremcounter}{theoremcounter}

\theoremstyle{plain}

\newtheorem{corollary}[theoremcounter]{Corollary}

\newtheorem{proposition}[theoremcounter]{Proposition}

\newtheorem{theorem}[theoremcounter]{Theorem}

\theoremnumbering{Alph}

\theoremstyle{definition}

\theoremstyle{remark}

\usepackage{xargs}                      
\usepackage[pdftex,dvipsnames]{xcolor}  
\usepackage[colorinlistoftodos,prependcaption,textsize=tiny]{todonotes}
\newcommandx{\unsure}[2][1=]{\todo[linecolor=red,backgroundcolor=red!25,bordercolor=red,#1]{#2}}
\newcommandx{\change}[2][1=]{\todo[linecolor=blue,backgroundcolor=blue!25,bordercolor=blue,#1]{#2}}
\newcommandx{\info}[2][1=]{\todo[linecolor=OliveGreen,backgroundcolor=OliveGreen!25,bordercolor=OliveGreen,#1]{#2}}
\newcommandx{\improvement}[2][1=]{\todo[linecolor=Plum,backgroundcolor=Plum!25,bordercolor=Plum,#1]{#2}}

\usepackage[nameinlink, capitalize, noabbrev]{cleveref}

\usepackage[style=alphabetic]{biblatex}
\DeclareFieldFormat{eprint}{\href{https://arxiv.org/abs/#1}{\texttt{#1}}}

\renewbibmacro*{in:}{}

\DeclareFieldFormat
[article,inbook,incollection,inproceedings,patent,thesis,unpublished,misc]
{title}{#1}

\renewbibmacro*{date}{%
\printdate
\iffieldundef{origyear}{%
}{%
  \setunit*{\addspace}%
  \printtext[parens]{\printorigdate}%
}%
}

\input{shortcuts.tex}

\newcommand{\authors}{Sven Raum}
\renewcommand{\title}{Twisted group C*-algebras of acylindrically hyperbolic groups have stable rank one}

\begin{document}


\thispagestyle{empty}

\noindent
\begin{minipage}{\linewidth}
  \begin{center}
    \textbf{\Large \title} \\
    \authors    
  \end{center}
\end{minipage}

\renewcommand{\thefootnote}{}
 \footnotetext{last modified on \today}

\vspace{2em}
\noindent
\begin{minipage}{\linewidth}
  \textbf{Abstract}.
  We prove that the twisted group C*-algebra of an acylindrically hyperbolic group -- not necessarily having trivial finite radical -- has stable rank one.
\end{minipage}


\section{Introduction}
\label{sec:introduction}

The (topological) stable rank of a C*-algebra is an important invariant introduced in \cite{rieffel83} and has applications to non-stable K-theory.  In this 1983 article, Rieffel asked whether the reduced group C*-algebras of free groups have stable rank one, meaning that the group of invertible elements $\mathrm{GL}(\Cstarred(\freegrp{n}))$ is dense in $\Cstarred(\freegrp{n})$.  This question was answered affirmatively by Dykema--Haagerup-R{\o}rdam in \cite{dykemahaageruprordam97}.  Their method was generalised to torsion-free hyperbolic groups by Dykema--de la Harpe in \cite{dykemadelaharpe99}.  More than a decade later, the class of acylindrically hyperbolic groups crystallised as a common generalisation of non-elementary hyperbolic groups, mapping class groups of most punctured closed surfaces and outer automorphism groups of non-abelian free groups. In a MathOverflow post in 2013 Osin announced a proof that the reduced group C*-algebras of acylindrically hyperbolic groups with trivial finite radical have stable rank one \cite{osin2013-acylindrically-hyperbolic-stable-rank}, pointing out the open problem whether \emph{every} acylindrically hyperbolic group has stable rank one.  Several years later, the announced result was published in \cite{gerasimovaosin20} still leaving the question open whether the hypothesis on the finite radical can be removed.  While this result covers in particular all hyperbolic groups with torsion elements but with trivial finite radical, which were not covered by \cite{dykemadelaharpe99}, the remaining hypothesis excludes examples such as $\mathrm{GL}(2, \ZZ) \cong \mathrm{Out}(\freegrp{2})$ as well as the mapping class groups of the surfaces $S_{0,2}, S_{1,0}, S_{1,1}, S_{1,2}$ and $S_{2,0}$, where $S_{g,n}$ is the closed surface of genus $g$ with $n$ punctures.  All these groups have non-trivial finite centre -- see \cite[Section 3.4]{farbmargalit2011} for the case of mapping class groups.  The main result of the present note is the following theorem which solves the open problem from \cite{osin2013-acylindrically-hyperbolic-stable-rank,gerasimovaosin20} covering additionally twisted group C*-algebras. It is even new for the class of hyperbolic groups whose finite radical is non-trivial.
\begin{theorem}
  \label{thm:acyclindrically-hyperbolic-stable-rank-one}
  Let $\Gamma$ be an acylindrically hyperbolic group and $\sigma \in \rZ^2(\Gamma)$.  Then $\Cstarred(\Gamma, \sigma)$ has stable rank one.
\end{theorem}

Our argument relies on classical decomposition results for twisted group C*-algebras combined with the following Proposition, transferring rapid-decay type estimates from untwisted to twisted group \mbox{C*-algebras}.
\begin{proposition}
  \label{prop:transfer-RD-estimate}
  Let $\Gamma$ be a group, $S \subseteq \Gamma$ a subset and $C > 0$.  Assume that for every $a \in \CC\Gamma$ with $\supp a \subseteq S$ we have $\|a\|_{\Cstarred(\Gamma)} \leq C \|a\|_2$.  Then for every 2-cocycle $\sigma \in \rZ^2(\Gamma, \rS^1)$ and every $a \in \CC\Gamma$ with $\supp a \subseteq S$ we have $\|a\|_{\Cstarred(\Gamma, \sigma)} \leq C \|a\|_2$.
\end{proposition}
Upon publication of this note, the author was informed that B{\'e}dos-Omland had made a similar observation in \cite[Appendix A, proof of Lemma A.2]{bedosomland2018-twisted}, inferring stable rank one for certain twisted group C*-algebras.  This result could alternatively be used in our proof of Theorem~\ref{thm:acyclindrically-hyperbolic-stable-rank-one}.  However, as the formulation of Proposition~\ref{prop:transfer-RD-estimate} also leads to the observation made in Corollary~\ref{cor:norm-normal-element}, we decided to keep it in its present form.

\subsection*{Acknowledgements}

The author thanks Hannes Thiel for inspiring conversations about the stable rank of group C*-algebras, motivating us to write down the present note. We also thank Dennis Osin for pointing out improvements to the introduction of an earlier version of this work.  We thank Jamie Bell and Tron Omland for pointing out the Appendix of \cite{bedosomland2018-twisted} and Tron Omland for pointing out \cite{bedos1991-simple-cstar-algebras}.

\section{Proofs}
\label{sec:proofs}

\begin{proof}[Proof of Proposition~\ref{prop:transfer-RD-estimate}]
  Fix a 2-cocycle $\sigma \in \rZ^2(\Gamma, \rS^1)$ and identify $\CC \Gamma$ with $\CC(\Gamma, \sigma)$ as vector spaces.  We denote by $*$ the convolution product in $\CC \Gamma$ and by $*_\sigma$ the convolution product in $\CC(\Gamma, \sigma)$.  Further, for $a = \sum_g a_g u_g \in \CC \Gamma$ we write $a^+ = \sum_g |a_g| u_g \in \CC \Gamma$.  For $a, b \in \CC \Gamma$ the $g$-th coefficient of the twisted convolution product satisfies
  \begin{gather*}
    (a *_\sigma b)^+_g
    =
    |\sum_{g = xy} \sigma(x,y)a_xb_y|
    \leq
    \sum_{g = xy} |\sigma(x,y)a_xb_y|
    =
    \sum_{g = xy} |a_x| \cdot |b_y|
    =
    (a^+ * b^+)_g
    \eqstop
  \end{gather*}
  Fix now $a \in \CC \Gamma$ with $\supp a \subseteq S$.  Then we conclude
  \begin{gather*}
    \|a\|_{\Cstarred(\Gamma, \sigma)}
    =
    \sup_{b \in \CC \Gamma \setminus \{0\}} \frac{ \|a *_\sigma b\|_2}{\|b\|_2}
    =
    \sup_{b \in \CC \Gamma \setminus \{0\}} \frac{ \|(a *_\sigma b)^+\|_2}{\|b^+\|_2}
    \leq
    \sup_{b \in \CC \Gamma \setminus \{0\}} \frac{ \|a^+ * b^+\|_2}{\|b^+\|_2}
    \leq
    C \|a\|_2
    \eqstop
  \end{gather*}
  This is what we had to show.
\end{proof}

There is essentially one strategy to prove stable rank one of group C*-algebras, which is based on R{\o}rdam's result \cite[Theorem 2.6]{rordam88} and originates in \cite{dykemahaageruprordam97}.  We use a generalisation of a formulation from \cite[Theorem 1.4]{dykemadelaharpe99}, which also applies to twisted group C*-algebras.  Given a group $\Gamma$ let us write $\rmr(a)$ for the spectral radius of $a \in \Cstarred(\Gamma)$. If further a 2-cocycle $\sigma \in \rZ^2(\Gamma, \rS^1)$ is given, we write $\rmr_\sigma(a)$ for the spectral radius of $a \in \Cstarred(\Gamma, \sigma)$.  Finally $\rmr_2(a) = \lim_n \|a^n\|_2^{1/n}$ for the $\ltwo$-spectral radius of $a \in \ltwo(\Gamma)$.
\begin{theorem}[Compare {\cite[Theorem 1.4]{dykemadelaharpe99}}]
  \label{thm:dykema-delaharpe}
  Let $\Gamma$ be a group and $\sigma \in \rZ^2(\Gamma, \rS^1)$ a 2-cocycle.  Assume that for any finite subset $F \subseteq \Gamma$ and there is $t \in \Gamma$ such that $tF$ generates a free subsemigroup and every $a$ in $\CC(\Gamma, \sigma)$ with $\supp(a) \subseteq tF$ satisfies $\rmr_\sigma(a) = \rmr_2(a)$.  Then $\Cstarred(\Gamma, \sigma)$ has stable rank one.
\end{theorem}

Combining Proposition~\ref{prop:transfer-RD-estimate} with the generalised version of Dykema--de la Harpe's criterion above, the next result follows from the same proof as \cite[Theorem 1.1]{gerasimovaosin20}.
\begin{corollary}
  \label{cor:acylindrically-hyperbolic-trivial-finite-radical-twisted}
  Let $\Gamma$ be an acylindrically hyperbolic group with trivial finite radical and let $\sigma \in \rZ^2(\Gamma, \rS^1)$ be a 2-cocycle.  Then the twisted group \Cstar-algebra $\Cstarred(\Gamma, \sigma)$ has stable rank one.
\end{corollary}

The corollary above together with Packer-Raeburn's theory of twisted crossed products and Green's imprimitivity theorem allow to obtain a positive answer to the problem raised by Osin \cite{osin2013-acylindrically-hyperbolic-stable-rank} and Gerasimova-Osin \cite{gerasimovaosin20}.
\begin{proof}[Proof of Theorem~\ref{thm:acyclindrically-hyperbolic-stable-rank-one}]
  Let $K \unlhd \Gamma$ be the finite radical and write $\Lambda = \Gamma/K$.  Then by \cite[Lemma 3.9]{minasyanosin2015-acylindrically-hyperbolic}  (see also \cite[Lemma 1]{minasyanosin2019-acylindrically-hyperbolic-correction}) it follows that $\Lambda$ is acylindrically hyperbolic with trivial finite radical.  After choosing a section $\rms \colon \Lambda \to \Gamma$ satisfying $\rms(e) = e$, we infer from \cite[Theorem 2.1]{bedos1991-simple-cstar-algebras} (see also \cite[Theorem 4.1]{packerraeburn89}) that there is an identification with a twisted crossed product $\Cstarred(\Gamma, \sigma) \cong \CC[K, \sigma] \rtimes_{\alpha, \rho, \mathrm{red}} \Lambda$.  Here the maps $\alpha \colon \Lambda \to \Aut(\CC[K, \sigma])$ and $\rho \colon \Lambda \times \Lambda \to \cU(\CC[K, \sigma])$  are given on $h, h_1, h_2 \in \Lambda$ and $k \in K$ by the formulas 
  \begin{align*}
    \alpha_h(u_k) & = \sigma(\rms(h), k) \sigma(\rms(h)  k \rms(h)^{-1}, \rms(h)) u_{\rms(h)k\rms(h)^{-1}} \\
    \rho(h_1,h_2) & = \sigma(\rms(h_1), \rms(h_2)) \ol{\sigma(\rms(h_1)\rms(h_2)\rms(h_1h_2)^{-1}, \rms(h_1h_2))} u_{\rms(h_1)\rms(h_2)\rms(h_1h_2)^{-1}}
    \eqstop
  \end{align*}
  Since $\CC[K, \sigma]$ is finite dimensional, it is a multi-matrix algebra.  So the twisted C*-dynamical system $(\CC[K, \sigma], \Lambda, \alpha, \rho)$ decomposes as a direct sum of $\Lambda$-simple dynamical systems.  By \cite[Theorem 2.13]{green78} each direct summand is isomorphic with 
  \begin{gather*}
    \Cstarred(\Upsilon, \nu) \otimes \rM_n(\CC) \otimes \rM_{\Gamma/\Upsilon}(\CC)
  \end{gather*}
  for a finite index subgroup $\Upsilon \leq \Lambda$ and a 2-cocyle $\nu \in \rZ^2(\Upsilon, \rS^1)$ depending on the direct summand.  As the class of acylindrically hyperbolic groups with trivial finite radical is stable under passage to finite index subgroups by \cite[Theorem 2.35]{dahmaniguirardelosin17}, we can apply Corollary~\ref{cor:acylindrically-hyperbolic-trivial-finite-radical-twisted} which shows that $\Cstarred(\Gamma, \sigma)$ is a finite direct sum of C*-algebras with stable rank one.  We conclude by remarking that stable rank one is stable under taking finite direct sums.
\end{proof}

As a byproduct of Proposition~\ref{prop:transfer-RD-estimate} we observe the following property, which to the best of our knowledge has not been previously observed for any class of groups.
\begin{corollary}
  \label{cor:norm-normal-element}
  Let $\Gamma$ be an acylindrically hyperbolic group with trivial finite radical and let $a \in \CC \Gamma$ be normal.  Then for every 2-cocycle $\sigma \in \rZ^2(\Gamma, \rS^1)$ we have $\|a\|_{\Cstarred(\Gamma, \sigma)} = \|a\|_{\Cstarred(\Gamma)}$.
\end{corollary}
\begin{proof}
  Using Proposition~\ref{prop:transfer-RD-estimate}, the proof of \cite[Corollary 5.5]{gerasimovaosin20} shows that the spectral radius of $a$ in both $\Cstarred(\Gamma)$ and $\Cstarred(\Gamma, \sigma)$ agrees with the $\ltwo$-spectral radius, that is $\rmr(a) = \rmr_2(a) = \rmr_\sigma(a)$.  As $a$ is normal, we can infer that $\|a\|_{\Cstarred(\Gamma, \sigma)} = \|a\|_{\Cstarred(\Gamma)}$ holds.
\end{proof}



{\small
  \printbibliography
}


\vspace{2em}

\noindent \begin{minipage}[t]{\linewidth}
  \small
  Sven Raum, Institute of Mathematics, University of Potsdam, Campus Golm, Haus 9, Karl-Liebknecht-Str. 24-25, Germany \\
  E-mail address: sven.raum@uni-potsdam.de
\end{minipage}

\end{document}

%% file: shortcuts.tex
\newcommand{\cU}{\ensuremath{\mathcal{U}}}

\newcommand{\rM}{\ensuremath{\mathrm{M}}}

\newcommand{\rS}{\ensuremath{\mathrm{S}}}

\newcommand{\rZ}{\ensuremath{\mathrm{Z}}}

\newcommand{\rmr}{\ensuremath{\mathrm{r}}}
\newcommand{\rms}{\ensuremath{\mathrm{s}}}

\newcommand{\ol}{\overline}

\newcommand{\eqstop}{\ensuremath{\, \text{.}}}

\newcommand{\ZZ}{\ensuremath{\mathbb{Z}}}

\newcommand{\CC}{\ensuremath{\mathbb{C}}}

\newcommand{\Aut}{\ensuremath{\mathrm{Aut}}}

\newcommand{\Cstar}{\ensuremath{\mathrm{C}^*}}

\newcommand{\supp}{\ensuremath{\mathop{\mathrm{supp}}}}

\newcommand{\Cstarred}{\ensuremath{\Cstar_\mathrm{red}}}

\newcommand{\ltwo}{\ensuremath{\ell^2}}

\newcommand{\freegrp}[1]{\ensuremath{\mathbb{F}}_{#1}}

